\documentclass[12pt, reqno]{amsart}

\usepackage[margin=1.1in]{geometry}

\usepackage{amssymb,upref}
\usepackage{amsmath, amsthm}
\usepackage{enumerate, graphicx, float}
\usepackage{relsize, mathrsfs, upgreek, bbm, eufrak}

\usepackage{mathtools}

\usepackage{color}
\usepackage{esint}

\theoremstyle{plain}

\newtheorem{thm}{Theorem}[section]
\newtheorem{lem}[thm]{Lemma}
\newtheorem{coro}[thm]{Corollary}

\theoremstyle{remark}

\newtheorem{remark}[thm]{Remark}
\numberwithin{equation}{section}

\def\XXint#1#2#3{{\setbox0=\hbox{$#1{#2#3}{\int}$ }
\vcenter{\hbox{$#2#3$ }}\kern-.6\wd0}}

\newcommand{\bee}{\begin{equation}}
\newcommand{\eee}{\end{equation}}
\newcommand{\be}{\begin{equation*}}
\newcommand{\ee}{\end{equation*}}

\newcommand{\eps}{\varepsilon}

\newcommand{\na}{\mathbb{N}}
\newcommand{\re}{\mathbb{R}}
\newcommand{\com}{\mathbb{C}}

\newcommand{\rn}{\mathbb{R}^n}

\newcommand{\sn}{\mathbb{S}^{n-1}}

\newcommand{\norm}[2]{\|#1\|_{#2}}

\newcommand{\dx}{\, dx}

\newcommand{\Ainf}{A_\infty(\rn)}

\DeclareMathOperator*{\essinf}{ess\,inf\,}

\begin{document}

\subjclass[2010]{Primary 42B37; Secondary 42B99.}

\keywords{Muckenhoupt weights, interpolation inequalities, weighted inequalities.}

\address{Diego Maldonado, Kansas State University, Department of Mathematics. 138 Cardwell Hall, Manhattan, KS-66506, USA.} \email{dmaldona@ksu.edu}

\address{Javier Soria, Departamento de An\'alisis Matem\'atico y Matem\'atica Aplicada, Fa\-cul\-tad de Matem\'aticas, Universidad Complutense de Madrid, Plaza de Ciencias 3, 28040 Madrid, Spain and ICMAT.}
\email{javier.soria@ucm.es}

\title[]{Anisotropic Caffarelli-Kohn-Nirenberg inequalities with higher-order fractional derivatives}
\author[Diego Maldonado and Javier Soria]{Diego Maldonado$^{*}$ and Javier Soria$^{**}$}

\thanks{$^{*}$Partially supported by Simons Foundation grant MPS-TSM-00007229.$^{**}$Partially supported by grants PID2024-155917NB-I00 and CEX-2023-001347-S, funded by MCIN/AEI/ 10.13039/501100011033, and Grupo UCM-970966.}

\date{\today}

\begin{abstract} We prove a family of anisotropic Caffarelli-Kohn-Nirenberg interpolation inequalities involving higher-order fractional derivatives and weights of the form $|x'|^{\theta_1}|x|^{\theta_2}|x_n|^{\theta_3}$.
\end{abstract}

\maketitle

\section{Introduction and main result}
In their foundational work \cite{CKN1}, L.~A.~Caffarelli, R. Kohn, and L. Nirenberg characterized the indices $1 \leq p, q \leq \infty$, $0 < r \leq \infty$, and exponents $\alpha, \beta, \gamma \in \re$ such that, for each $\theta \in [0,1]$ the first-order weighted interpolation inequality
\begin{equation}\label{CKN}
\norm{|x|^\gamma f}{L^r} \lesssim \norm{|x|^\alpha |D f|}{L^p}^\theta  \norm{|x|^\beta f}{L^q}^{1-\theta}
\end{equation}
holds true for every $f$ in the Schwartz class $\mathcal{S}(\rn)$. Later on, in \cite{Lin86} C.-S. Lin extended \eqref{CKN} as to include higher-order derivatives by providing necessary and sufficient conditions on $1 \leq p, q \leq \infty$, $0 < r \leq \infty$, $\alpha, \beta, \gamma \in \re$, $j \in \na_0$, $m \in \na$, with $j \leq m$, for the validity of the higher-order weighted interpolation inequality
\begin{equation*}
\norm{|x|^\gamma |D^j f|}{L^r} \lesssim \norm{|x|^\alpha |D^m f|}{L^p}^\theta  \norm{|x|^\beta f}{L^q}^{1-\theta}, \quad \forall f \in \mathcal{S}(\rn), \theta \in [j/m, 1].
\end{equation*}
More recently, in \cite[Theorem 1.6]{DD23} R. Duarte and J. Drumond Silva extended Lin's result by identifying indices $1 < p, q, r < \infty$; $0 \leq t < s$ (smoothness indices, not necessarily integers), $t/s \leq \theta \leq 1$; $\alpha, \beta, \gamma \in \re$, such that the higher-order fractional-derivative weighted interpolation inequality
$$
\norm{|x|^\gamma |D^t f|}{L^r} \lesssim \norm{|x|^\alpha |D^s f|}{L^p}^\theta  \norm{|x|^\beta f}{L^q}^{1-\theta}, 
$$
holds true for every $f \in \mathcal{S}(\rn)$. When $n \geq 2$ and in the context of asymptotic stability of solutions to the Navier-Stokes equations, in \cite[Theorem 1.1]{LiYan} Y. Y.  Li and X. Yan introduced an anisotropic version of the Caffarelli-Kohn-Nirenberg inequality \eqref{CKN} by determining the indices  $1 \leq p, q \leq \infty$, $0 < r \leq \infty$, and exponents $\gamma_1, \gamma_2, \gamma_3, \alpha, \mu, \beta$ such that 
\begin{equation*}
\norm{|x|^{\gamma_1} |x'|^{\alpha} f}{L^r} \lesssim \norm{|x|^{\gamma_2} |x'|^{\mu} |\nabla f|}{L^p}^\theta  \norm{|x|^{\gamma_3} |x'|^{\beta} f}{L^q}^{1-\theta}, \quad \forall f \in \mathcal{S}(\rn), \theta \in [0, 1],
\end{equation*}
where $x=(x', x_n) \in \rn$ and $x':=(x_1,\ldots, x_{n-1})$. Always with $n \geq 2$, in their study on the regularity of solutions to $p$-Laplace equations, in \cite{MZarXiv1} C. X. Miao and Z. W. Zhao identified indices $1 < p < q$ and exponents $\theta_1, \theta_2, \theta_3 \in \re$ yielding the anisotropic Sobolev inequality 
\begin{equation}\label{Sob:MZ}
\norm{f}{L^q(u_{\theta_1, \theta_2, \theta_3})} \lesssim \norm{\nabla f}{L^p(u_{\theta_1, \theta_2, \theta_3})}, \quad \forall f \in \mathcal{S}(\rn),
\end{equation}
where 
\begin{equation}\label{def:u:thetas}
u_{\theta_1, \theta_2, \theta_3}(x):= |x'|^{\theta_1}|x|^{\theta_2}|x_n|^{\theta_3}, 
\end{equation}
see \cite[Corollary 2.9]{MZarXiv1}. As a key element in their approach, given $1 < p < \infty$, Miao and Zhao introduced in \cite{MZarXiv1} the exponent classes
\begin{align*}
\begin{cases}
&\mathcal{A}:=\{(\theta_{1},\theta_{2},\theta_{3}): \theta_{1}>-(n-1),\,\theta_{2}\geq0,\,\theta_{3}>-1\},\\
&\mathcal{B}:=\{(\theta_{1},\theta_{2},\theta_{3}):\theta_{1}>-(n-1),\,\theta_{2}<0,\,\theta_{3}>-1,\,\theta_{1}+\theta_{2}+\theta_{3}>-n\},\\
&\mathcal{C}_{p}:=\{(\theta_{1},\theta_{2},\theta_{3}):\theta_{1}<(n-1)(p-1),\,\theta_{2}\leq0,\,\theta_{3}<p-1\},\\
&\mathcal{D}_{p}:=\{(\theta_{1},\theta_{2},\theta_{3}):\theta_{1}<(n-1)(p-1),\,\theta_{2}>0,\,\theta_{3}<p-1,\,\theta_{1}+\theta_{2}+\theta_{3}<n(p-1)\},
\end{cases}
\end{align*}
to describe the membership of the weight $u_{\theta_1, \theta_2, \theta_3}$ as defined in \eqref{def:u:thetas} to the Muckenhoupt class $A_p(\rn)$. More precisely, by \cite[Theorem 3.2 ]{MZarXiv1}, $u_{\theta_1, \theta_2, \theta_3} \in A_p(\rn)$ if and only if $(\theta_{1},\theta_{2},\theta_{3})\in (\mathcal{A}\cup\mathcal{B})\cap(\mathcal{C}_{p}\cup\mathcal{D}_{p})$, see Section \ref{secc:prelim} below. On the other hand, as proved in Lemma \ref{lemma:L1loc:indices} in Section \ref{secc:prelim}, $(\theta_1, \theta_2, \theta_3) \in \mathcal{A}\cup\mathcal{B}$ means exactly that $u_{\theta_1, \theta_2, \theta_3} \in L^1_{\rm{loc}}(\rn)$. For the sake of coherence, we have kept the notations used in \cite{MZarXiv1}, although it is easy to prove that, in fact:
\begin{align*}
\mathcal{A}\cup\mathcal{B}&=\{(\theta_{1},\theta_{2},\theta_{3}):  \theta_{1}>-(n-1),\,\theta_{3}>-1,\,\theta_{1}+\theta_{2}+\theta_{3}>-n\},\\
\mathcal{C}_p\cup\mathcal{D}_p&=\{(\theta_{1},\theta_{2},\theta_{3}):  \theta_{1}<(n-1)(p-1),\,\theta_{3}<p-1,\,\theta_{1}+\theta_{2}+\theta_{3}<n(p-1)\}.
\end{align*}

The purpose of this work is to provide a full scale of anisotropic interpolation inequalities based on weights of the form $u_{\theta_1, \theta_2, \theta_3}$ and involving the homogeneous higher-order fractional derivative operator $D^s$ (defined as $\widehat{D^sf}:= |\cdot|^s \hat{f}$ for $s > 0$). This scale of anisotropic interpolation inequalities extends all the aforementioned ones. Our main results are the following Theorems \ref{thm:t/s}, \ref{thm:theta=1}, and \ref{thm:main}.

\begin{thm}\label{thm:t/s} Fix $1 < p, q, r < \infty$ and $0 < t < s$. Set $\theta:=t/s \in (0,1)$ and suppose that 
\begin{equation}\label{rel:pqrst}
\frac{1}{r}= \frac{\theta}{p} +\frac{(1 - \theta)}{q}.
\end{equation}
Let $\alpha_0, \alpha_1, \beta_0,  \beta_1, \gamma_0, \gamma_1 \in \re$ satisfy 
\begin{align}\label{p:1:p}
 p(\alpha_1, \beta_1, \gamma_1) & \in (\mathcal{A}\cup\mathcal{B})\cap(\mathcal{C}_{p}\cup\mathcal{D}_{p}),\\\label{q:0:q}
q (\alpha_0, \beta_0, \gamma_0)  & \in (\mathcal{A}\cup\mathcal{B})\cap(\mathcal{C}_{q}\cup\mathcal{D}_{q}), and\\\label{aj:bk:gl:L1loc}
(\alpha_j, \beta_k, \gamma_\ell) & \in \mathcal{A}\cup\mathcal{B}, \quad \text{ for } j,k,\ell =0,1. 
\end{align}
Then,
\begin{equation}\label{GN:theta:t/s}
\norm{|x'|^{\alpha_\theta}|x|^{\beta_\theta}|x_n|^{\gamma_\theta}   D^t f}{L^r} \lesssim \norm{|x'|^{\alpha_1}|x|^{\beta_1}|x_n|^{\gamma_1}  D^s f}{L^p}^{\theta} \norm{|x'|^{\alpha_0}|x|^{\beta_0}|x_n|^{\gamma_0}  f}{L^q}^{1-\theta},
\end{equation}
for every $f \in \mathcal{S}(\rn)$, where $(\alpha_\theta, \beta_\theta, \gamma_\theta) :=  (\alpha_1, \beta_1, \gamma_1) \theta +  (\alpha_0, \beta_0, \gamma_0)(1-\theta)$. 
\end{thm}

For the statement of our next theorems, given $1 < p \leq q < \infty$ and $0 < \alpha < n$ define
\begin{equation}\label{def:alpha0}
\Theta_{p,q}(\alpha):= \alpha + n \left( \frac{1}{q} - \frac{1}{p}\right).
\end{equation}
As pointed out in Remark \ref{rmk:Theta:>=0} below, we can assume $\Theta_{p,q}(\alpha) \geq 0$. The cases $\Theta_{p,q}(\alpha) =0$ (that is, $q = np/(n-\alpha p)$) and $\Theta_{p,q}(\alpha) >0$ (that is, $q < np/(n-\alpha p)$) are referred to as the \emph{critical} and \emph{subcritical} cases, respectively, for the indices $p, q, \alpha$. Our next two theorems cover both the critical and subcritical cases. Firstly, we have the following weighted Sobolev inequality, which corresponds to the case $\theta =1$ in the interpolation inequality.

\begin{thm}\label{thm:theta=1} Fix $1 < p \leq q < \infty$, $t <s < t +n$, and let $\theta_1, \theta_2, \theta_3, \sigma_1, \sigma_2, \sigma_3 \in \re$ satisfy
\begin{equation}\label{sigma:theta:Ainf}
-\frac{p'}{p}(\sigma_1, \sigma_2, \sigma_3), (\theta_1, \theta_2, \theta_3) \in \mathcal{A}\cup\mathcal{B}.
\end{equation}
\begin{enumerate}[(i)]
\item\label{Theta =0} If $\Theta_{p,q}(s-t):= s - t + n \left( \frac{1}{q} - \frac{1}{p}\right) = 0$ and 
\begin{equation}\label{theta:j:q:sigma:j:p}
\frac{\theta_j}{q} = \frac{\sigma_j}{p}, \quad \text{for } j=1,2,3;
\end{equation}
\item\label{Theta >0} or if $\Theta_{p,q}(s-t):= s - t + n \left( \frac{1}{q} - \frac{1}{p}\right) >0$ and
\begin{align}\label{sum:sigmas:thetas:alpha}
\frac{\theta_1+ \theta_2 + \theta_3}{q} -\frac{\sigma_1+ \sigma_2 + \sigma_3}{p} &=- \Theta_{p,q}(s-t),\\\label{sum:sigmas:thetas:alpha:2}
\left(\frac{\theta_1}{q }- \frac{\sigma_1}{p}\right)& > -\frac{(n-1)}{n}\Theta_{p,q}(s-t),\\\label{sum:sigmas:thetas:alpha:3}
\left(\frac{\theta_3}{q }- \frac{\sigma_3}{p}\right) &> -\frac{1}{n}\Theta_{p,q}(s-t);
\end{align}
\end{enumerate}
then, 
\begin{equation}\label{W:Sob:thetas:sigmas}
\norm{D^t f}{L^q(u_{\theta_1, \theta_2, \theta_3})} \lesssim \norm{D^s f}{L^p(u_{\sigma_1, \sigma_2, \sigma_3})}, \quad \forall f \in D^{>\tau} \mathcal{S}(\rn),
\end{equation}
with $\tau:= \max\{-n, -t\}$, $
D^{> \tau}:=\{f \in L^\infty: f = D^bg\text{ for some } g \in \mathcal{S}(\rn) \text{ and } b > \tau\},
$
and the implicit constant depends only on $p$, $q$, $s,$ $t,$ $\theta_1, \theta_2,$ $\theta_3, \sigma_1, \sigma_2, \sigma_3$, and $n$.  
\end{thm}

\begin{remark}\label{rmk:weights:inside} As pointed out in \cite[Remark 3]{DD23}, if $t>0$ then $\tau <0$ and $D^{> \tau}$ contains the Schwartz class $\mathcal{S}(\rn)$. In particular, the weighted Sobolev inequality \eqref{W:Sob:thetas:sigmas} extends \eqref{Sob:MZ} to the case of arbitrary fractional derivatives $D^t$ and $D^s$ with $0< t <s < t +n$. Finally, notice that \eqref{W:Sob:thetas:sigmas} can be written as
\begin{equation}\label{W:Sob:thetas:sigmas:2}
\norm{u_{\theta_1/q, \theta_2/q, \theta_3/q} D^t f}{L^q} \lesssim \norm{u_{\sigma_1/p, \sigma_2/p, \sigma_3/p} D^s f}{L^p}, \quad \forall f \in D^{>\tau} \mathcal{S}(\rn).
\end{equation}
\end{remark}

Secondly, for an arbitrary $\theta \in (t/s, 1)$ we have

\begin{thm}\label{thm:main}  Fix $1 < p, q, r < \infty$, $0 < t < s$, and $t/s < \theta < 1$,  and let $a \in (1, \infty)$ be defined by 
\begin{equation}\label{rel:apqtheta}
\frac{1}{a}:= \frac{\theta}{p} + \frac{1 -\theta}{q}.
\end{equation}
Suppose that $a \leq r$, that $\theta s < t +n$, and let $\alpha_0, \alpha_1, \beta_0,  \beta_1, \gamma_0, \gamma_1 \in \re$ satisfy 
\begin{align*}
p(\alpha_1, \beta_1, \gamma_1)  & \in (\mathcal{A}\cup\mathcal{B})\cap(\mathcal{C}_{p}\cup\mathcal{D}_{p}),\\
q(\alpha_0, \beta_0, \gamma_0)   & \in (\mathcal{A}\cup\mathcal{B})\cap(\mathcal{C}_{q}\cup\mathcal{D}_{q}),\\
(\alpha_j, \beta_k, \gamma_\ell) & \in \mathcal{A}\cup\mathcal{B}, \quad \text{ for } j,k,\ell =0,1. 
\end{align*}
As before, 
$
(\alpha_\theta, \beta_\theta, \gamma_\theta) :=  (\alpha_1, \beta_1, \gamma_1) \theta +  (\alpha_0, \beta_0, \gamma_0)(1-\theta).
$
\begin{enumerate}[(i)] 
\item If $\Theta_{a,r}(\theta s - t):=\theta s - t + n \left( \frac{1}{r} - \frac{1}{a}\right) =0$ and
\begin{equation}\label{cond:delta:gamma:1:2:B:1}
- a'(\alpha_\theta, \beta_\theta, \gamma_\theta), r (\alpha_\theta, \beta_\theta, \gamma_\theta)  \in \mathcal{A}\cup\mathcal{B}, \end{equation}
then
\begin{equation}\label{GN:E:0}
\norm{|x'|^{\alpha_\theta}|x|^{\beta_\theta}|x_n|^{\gamma_\theta} D^t f}{L^r} \lesssim \norm{|x'|^{\alpha_1}|x|^{\beta_1}|x_n|^{\gamma_1} D^s f}{L^p}^{\theta} \norm{|x'|^{\alpha_0}|x|^{\beta_0}|x_n|^{\gamma_0} f}{L^q}^{1-\theta}, 
\end{equation}
for every $f \in \mathcal{S}(\rn)$.
\item If $\Theta_{a,r}(\theta s - t):= \theta s - t + n \left( \frac{1}{r} - \frac{1}{a}\right) >0$ and if $\alpha_\theta', \beta_\theta', \gamma_\theta' \in \re$ satisfy
\begin{align}\label{sum:sigmas:thetas:alpha:a:r}
(\alpha_\theta' - \alpha_\theta) + (\beta_\theta' - \beta_\theta) + (\gamma_\theta' - \gamma_\theta) & = -\Theta_{a,r}(\theta s - t),\\\label{sum:sigmas:thetas:alpha:a:r:2}
\alpha_\theta'  -\alpha_\theta &> - \frac{(n-1)}{n}\Theta_{a,r}(\theta s - t),\\\label{sum:sigmas:thetas:alpha:a:r:3}
\gamma_\theta'  - \gamma_\theta &> - \frac{1}{n}\Theta_{a,r}(\theta s - t);
\end{align}
\end{enumerate}
then
\begin{equation}\label{GN:E:1}
\norm{|x'|^{\alpha_\theta'}|x|^{\beta_\theta'}|x_n|^{\gamma_\theta'} D^t f}{L^r} \lesssim \norm{|x'|^{\alpha_1}|x|^{\beta_1}|x_n|^{\gamma_1} D^s f}{L^p}^{\theta} \norm{|x'|^{\alpha_0}|x|^{\beta_0}|x_n|^{\gamma_0} f}{L^q}^{1-\theta}, 
\end{equation}
for every $f \in \mathcal{S}(\rn)$.
\end{thm}

In the case $a=r$ from Theorem \ref{thm:main}, the fact that $\Theta_{a,r}(\theta s - t) = \theta s - t >0$ immediately yields the following corollary. 

\begin{coro}\label{coro:a=r} Fix $1 < p, q < \infty$, $0 < t < s$, and $t/s < \theta < 1$ with $\theta s < t +n$. Let $r \in (1, \infty)$ be defined by 
\begin{equation}\label{rel:apqtheta:r}
\frac{1}{r}:= \frac{\theta}{p} + \frac{1 -\theta}{q}.
\end{equation}
Suppose that $\alpha_0, \alpha_1, \beta_0,  \beta_1, \gamma_0, \gamma_1 \in \re$ satisfy 
\begin{align*}
p(\alpha_1, \beta_1, \gamma_1)  & \in (\mathcal{A}\cup\mathcal{B})\cap(\mathcal{C}_{p}\cup\mathcal{D}_{p}),\\
q(\alpha_0, \beta_0, \gamma_0)   & \in (\mathcal{A}\cup\mathcal{B})\cap(\mathcal{C}_{q}\cup\mathcal{D}_{q}),\\
(\alpha_j, \beta_k, \gamma_\ell) & \in \mathcal{A}\cup\mathcal{B}, \quad \text{ for } j,k,\ell =0,1. 
\end{align*}
Fix $\alpha_\theta', \beta_\theta', \gamma_\theta' \in \re$ such that 
\begin{align*}
 (\alpha_\theta' - \alpha_\theta) + (\beta_\theta' - \beta_\theta) + (\gamma_\theta' - \gamma_\theta) & =- (\theta s - t),\\
\alpha_\theta'  -\alpha_\theta &> - \frac{(n-1)}{n}(\theta s - t),\\
\gamma_\theta'  - \gamma_\theta &> - \frac{1}{n}(\theta s - t).
\end{align*}
Then
$$
\norm{|x'|^{\alpha_\theta'}|x|^{\beta_\theta'}|x_n|^{\gamma_\theta'} D^t f}{L^r} \lesssim \norm{|x'|^{\alpha_1}|x|^{\beta_1}|x_n|^{\gamma_1} D^s f}{L^p}^{\theta} \norm{|x'|^{\alpha_0}|x|^{\beta_0}|x_n|^{\gamma_0} f}{L^q}^{1-\theta}, 
$$
for every $f \in \mathcal{S}(\rn)$.
\end{coro}

The rest of the article is organized as follows: Section \ref{secc:prelim} describes the relevant weight classes, sufficient criteria for membership in those classes, and basic facts on the anisotropic weight $u_{\theta_1, \theta_2, \theta_3}$. Section \ref{sec:approach} lays out an abstract, unifying approach towards weighted inequalities; whereas Sections \ref{sec:proof:thm:t/s}, \ref{secc:theta:1}, and  \ref{sec:inter:theta} contain the proofs of Theorems \ref{thm:t/s}, \ref{thm:theta=1}, and \ref{thm:main}, respectively. 
 
\section{Preliminaries}\label{secc:prelim}

\subsection{Weight classes}\label{sec:Ap:weights}  For $1 < p < \infty$ a weight $w$ in $\rn$ (that is, $w \in L^1_{\rm{loc}}(\rn)$ with $w \geq 0$ a.e. in $\rn$) is said to belong to the \emph{Muckenhoupt class} $A_p(\rn)$ if 
\begin{equation}\tag*{$(A_p)$}\label{Ap}
[w]_{A_p}:= \sup\limits_{Q} \left(\fint_Q w \right)  \left(\fint_Q w^{\frac{-1}{(p-1)}} \right)^{p-1} < \infty
\end{equation}
where $Q \subset \rn$ is a cube and $\fint_Q u$ denotes the average $\frac{1}{|Q|} \int_Q u$. For $p=1$ and $p=\infty$,
\begin{equation}\tag*{$(A_1)$}\label{def:A1}
[w]_{A_1}:= \sup\limits_{Q} \left(\fint_Q w \right)  \left(\essinf_Q w \right)^{-1} < \infty
\end{equation}
and
\begin{equation}\tag*{$(A_\infty)$}\label{def:Ainf}
[w]_{A_\infty}:= \sup\limits_{Q} \left(\fint_Q w \right)  \exp \left(-\fint_Q \ln w \right) < \infty.
\end{equation}
It follows that $A_\infty(\rn) = \bigcup_{p\geq1}A_p(\rn)$ and, for $1 < p < \infty$, 
\begin{equation}\label{w:Ap:Ainf}
w \in A_p(\rn) \Leftrightarrow w, w^{-p'/p} \in \Ainf,
\end{equation}
(see for instance \cite[Section~9.3 and Exercise 9.3.3]{GrafakosBook}). As typical examples of $A_p$-weights, for $1 < p < \infty$, the weight $|x|^a \in A_p(\rn)$ if and only if $-n < a < n (p-1)$, whereas $|x|^a \in A_1(\rn)$ if and only if $-n < a \leq 0$ (see for instance  \cite[Section 9.1.2]{GrafakosBook}). 

For $1 < p, q < \infty$ and $\alpha \in \re$, two (measurable) nonnegative functions $v$ and $w$ are said to satisfy the $A_{p,q}^\alpha$ condition (see \cite[Definition 2.3]{DD23}), in symbols $(v,w) \in A_{p,q}^\alpha$, if
\begin{equation*}
[(v,w)]_{A_{p,q}^\alpha}:= \sup\limits_Q |Q|^{\alpha/n -1} \left( \int_Q w\right)^{1/q} \left(\int_Q v^{-p'/p} \right)^{1/p'}  < \infty.
\end{equation*}
From the definition of $\Theta_{p,q}(\alpha)$ in \eqref{def:alpha0}, the condition $(v,w) \in A_{p,q}^\alpha$ can be written as 
\begin{equation}\label{def:Apq:Theta}
[(v,w)]_{A_{p,q}^\alpha}:= \sup\limits_Q |Q|^{\Theta_{p,q}(\alpha)/n} \left( \fint_Q w\right)^{1/q} \left(\fint_Q v^{-p'/p} \right)^{1/p'}  < \infty.
\end{equation}
In the case $\Theta_{p,q}(\alpha)=0$, let us define $(v,w) \in A_{p,q}$ (as in \cite[Definition 1.5]{DD23}) by 
\begin{equation}\label{def:Apq}
[(v,w)]_{A_{p,q}}:= \sup\limits_Q \left( \fint_Q w\right)^{1/q} \left(\fint_Q v^{-p'/p} \right)^{1/p'}  < \infty.
\end{equation}

\begin{remark}\label{rmk:Theta:>=0} If $\Theta_{p,q}(\alpha) < 0$, then the condition \eqref{def:Apq:Theta} along with the Lebesgue Differentiation Theorem yields $w \equiv 0$. Thus, it is only meaningful to consider $\Theta_{p,q}(\alpha) \geq 0$. 
\end{remark}

The following two results from \cite{MaSo} provide sufficient conditions on a pair $(v, w)$ with $w, v^{-p'/p} \in \Ainf$ so that $(v, w) \in A_{p,q}^\alpha$ within the critical and subcritical cases for the indices. For the critical case $\Theta_{p,q}(\alpha) = 0$, that is, $1/q = 1/p - \alpha/n$, we have

\begin{thm}[Theorem 1.1 from \cite{MaSo}]\label{thm:SW:Ainf:p:q:alpha0=0} Fix $1 < p \leq q < \infty$ and $0 < \alpha < n$ such that $\Theta_{p,q}(\alpha) = 0$. Given $(v, w)$ with $w, v^{-p'/p} \in \Ainf$ suppose that there exists $C_0 > 0$ with
\begin{equation}\label{w:1/q:C:v:1/p}
w(x)^{1/q} \leq C_0 \, v(x)^{1/p}, \quad \text{a.e. } x \in \rn.
\end{equation}
Then \eqref{def:Apq:Theta} holds true with the estimate $[(v, w)]_{A_{p,q}^\alpha} \leq  C_0 C_1$, where $C_1 > 0$ depends only on $n$, $p$, $q$, $\alpha$, $[w]_{\Ainf}$, and $[v^{-p'/p}]_{\Ainf}$.
\end{thm}

For the subcritical case $\Theta_{p,q}(\alpha) > 0$, that is, $1/q > 1/p - \alpha/n$, we have

\begin{thm}[Theorem 1.2 from \cite{MaSo}]\label{thm:SW:Ainf:p:q} Fix $1 < p \leq q < \infty$ and $0 < \alpha < n$ such that $\Theta_{p,q}(\alpha) > 0$. Given $(v, w)$ with $w, v^{-p'/p} \in \Ainf$ the condition $w^{1/q}/v^{1/p} \in L^{n/\Theta_{p,q}(\alpha), \infty}(\rn)$ implies \eqref{def:Apq:Theta} with the estimate
\begin{equation}\label{wv:pq:thm:SW}
[(v, w)]_{A_{p,q}^{\alpha}} \leq C_2 \norm{w^{1/q}/v^{1/p}}{L^{n/\Theta_{p,q}(\alpha), \infty}(\rn)},
\end{equation}
where $C_2 > 0$ depends only on $n$, $\alpha$, $p$, $q$, $[w]_{\Ainf}$, and $[v^{-p'/p}]_{\Ainf}$. 
\end{thm}

\subsection{On the anisotropic weights $u_{\theta_1, \theta_2, \theta_3}$}

Given $\theta_1, \theta_2, \theta_3 \in \re$, let $u_{\theta_1, \theta_2, \theta_3}$  be the anisotropic weight defined in \eqref{def:u:thetas}.

\begin{lem}\label{lemma:L1:w} Given $\theta_1, \theta_2, \theta_3 \in \re$ and $0< p<\infty$, then $u_{\theta_1, \theta_2, \theta_3} \in L^{p, \infty}(\rn)$ if and only if $\theta_1+\theta_2+ \theta_3 = -n/p$,  $\theta_1>-(n-1)/p$, and $\theta_3>-1/p$.

\end{lem}
 
\begin{proof} Since $u_{\theta_1, \theta_2, \theta_3} \in L^{p, \infty}(\rn)$ if and only if $u_{\theta_1, \theta_2, \theta_3} ^p\in L^{1, \infty}(\rn)$, it suffices to consider the case $p=1$. We want to characterize when the distribution function of  $u_{\theta_1, \theta_2, \theta_3}$ satisfies:
\begin{equation}\label{minusone}
\Big|\big\{x\in\mathbb R^n:u_{\theta_1, \theta_2, \theta_3}(x)>t\big\}\Big|\lesssim t^{-1}.
\end{equation}
For this purpose, we split this level set in two parts: when $|x'|\le|x_n|$, in which case $|x|\approx~|x_n|$ and hence $u_{\theta_1, \theta_2, \theta_3}(x)\approx |x'|^{\theta_1}|x_n|^{\theta_2+\theta_3}$, or $|x_n|\le|x'|$, getting now that $|x|\approx~|x'|$ and  $u_{\theta_1, \theta_2, \theta_3}(x)\approx |x'|^{\theta_1+\theta_2}|x_n|^{\theta_3}$. Thus, to simplify the notations, we consider the function $u(x)=|x'|^\alpha|x_n|^\beta$, with $\alpha,\beta\in\mathbb R$, and our goal now is to estimate the following integrals:
\begin{equation}\label{twointegrals}
I=\int_0^\infty\bigg(\int_{\big\{|x'|\le r:|x'|^\alpha r^\beta>t\big\}}dx'\bigg)\,dr\quad \text{and}\quad II=\int_{\mathbb R^{n-1}}\bigg(\int_{\big\{r\le |x'|:|x'|^\alpha r^\beta>t\big\}}dr\bigg)\,dx'.
\end{equation}
Set $s:=\alpha+\beta$. We start with $I$ and distinguish several cases:
\medskip

If $\alpha>0$ and $s\ge0$, it is easy to see that $I=\infty$. Assume now that $\alpha>0$ and $s<0$. Then 
\begin{equation}\label{onealphapos}
I\approx  \int_0^\infty\bigg(\int_{\big\{|x'|\le r:|x'|>\big(\frac{t}{r^\beta}\big)^{1/\alpha}\big\}}dx'\bigg)\,dr\approx
\int_0^{t^{1/s}}\bigg(r^{n-1}-\bigg(\frac{t}{r^\beta}\bigg)^{(n-1)/\alpha}\bigg)\,dr.
\end{equation}
The first term gives $t^{n/s}$ and \eqref{minusone} implies the condition $s=-n$. For the second term, we need that
$$
-\frac{\beta}{\alpha}(n-1)+1>0\iff \beta<\frac{s}{n}=-1,
$$
and hence 
$$
\int_0^{t^{1/s}}\bigg(\frac{t}{r^\beta}\bigg)^{(n-1)/\alpha}\,dr\approx t^{n/s}=t^{-1}.
$$
Therefore the integral in \eqref{onealphapos} is then comparable to $t^{-1}$ if and only if $s=-n$ and $\beta<-1$.
\medskip

If $\alpha=0$, then $s=\beta$. As in the previous case, it is easy to see that $I=\infty$ when $s\ge0$. Finally, if $s<0$, 
$$
I\approx\int_0^\infty\bigg(\int_{\big\{|x'|\le r:r<t^{1/s}\big\}}dx'\bigg)\,\approx\int_0^{t^{1/s}}r^{n-1}\approx t^{n/s}\approx t^{-1}\iff s=-n.
$$
\medskip

If $\alpha<0$, again $I=\infty$ when $s\ge0$.  If $s<0$,
$$
I\approx \int_0^{t^{1/s}}r^{n-1}\,dr+\int_{t^{1/s}}^\infty\bigg(\frac{t}{r^\beta}\bigg)^{(n-1)/\alpha}\,dr
$$
The first term is comparable to $t^{-1}$ if and only if $s=-n$ and the second term in integrable if and only if $-\frac{\beta}{\alpha}(n-1)+1<0\iff\beta<-1$, and the integral is comparable to $t^{n/s}=t^{-1}$. Therefore, putting all these different cases together, we conclude that 
\begin{equation}\label{firstintegral}
\begin{gathered}
I\approx\Big|\bigl\{x\in\mathbb R^n:|x'|\le|x_n|\text{ and } |x'|^\alpha|x_n|^\beta>t\bigr\}\Big|\lesssim t^{-1} \\
\Updownarrow \\
\alpha+\beta=-n\text{ and } \beta<-1.
\end{gathered}
\end{equation}
We now proceed to estimate the integral $II$ in \eqref{twointegrals}. Assume first that $\beta >0$. Again, when $s=0$ it turns out that $II=\infty$. When $s<0$,
$$
II\approx\int_{\{|x'|<t^{1/s}\}}\bigg(|x'|-\bigg(\frac{t}{|x'|^\alpha}\bigg)^{1/\beta}\bigg)\,dx'.
$$
The first term gives $t^{n/s}\approx t^{-1}\iff s=-n$, and the second integral is finite if and only if $-\frac{\alpha}{\beta}+n-1>0\iff\beta>-1$, in which case, the integral is equal to $t^{1/\beta+(-\alpha/\beta+n-1)/s}=t^{-1}$.
\medskip
It is easy to see that $II=\infty$ whenever $\beta=0$ and $\alpha\ge0$. If $\beta =0$ and $\alpha<0$, 
$$
II\approx\int_{\{|x'|<t^{1/\alpha}\}}|x'|\,d x'\approx t^{n/\alpha}= t^{-1}\iff s=\alpha=-n.
$$
Finally, if $\beta<0$ and $s\ge0$, the integral  $II$ is equal to infinity, and if $\beta<0$ and $s<0$,
$$
II\approx \int_{\{|x'|<t^{1/s}\}}|x'|\,dx'+\int_{\{|x'|\ge t^{1/s}\}}\bigg(\frac{t}{|x'|^\alpha}\bigg)^{1/\beta}\,dx'.
$$
The first term is comparable to $t^{n/s}$, which again gives that $s=-n$, while the second integral is finite if and only if $-\alpha/\beta+n-1<0\iff \beta>-1$, and hence the integral is comparable to $t^{{1/\beta}+(-\alpha/\beta+n-1)/s}=t^{-1}.$ Therefore, by putting all these different cases together, we conclude that 
\begin{equation}\label{secondintegral}
\begin{gathered}
II\approx\Big|\bigl\{x\in\mathbb R^n:|x_|\le|x'|\text{ and } |x'|^\alpha|x_n|^\beta>t\bigr\}\Big|\lesssim t^{-1} \\
\Updownarrow \\
\alpha+\beta=-n\text{ and } \beta>-1.
\end{gathered}
\end{equation}
By applying \eqref{firstintegral} and \eqref{secondintegral} to the function  $u_{\theta_1, \theta_2, \theta_3}$, for the integral $I$ we have that $\alpha=\theta_1$ and $\beta=\theta_2+\theta_3$ and similarly, for the second integral $II$, $\alpha=\theta_1+\theta_2$ and $\beta=\theta_3$. Hence,
\begin{align*}
u_{\theta_1, \theta_2, \theta_3}\in L^{1,\infty}(\mathbb R^n)&\iff \theta_1+\theta_2+\theta_3=-n,\ \theta_2+\theta_3<-1,\text{ and }\theta_3>-1\\
&\iff \theta_1+\theta_2+\theta_3=-n,\ \theta_1>-(n-1),\text{ and }\theta_3>-1.
\end{align*}
\end{proof}

\begin{lem}\label{lemma:L1loc:indices} $u_{\theta_1, \theta_2, \theta_3} \in L^1_{\rm{loc}}(\rn)$ if and only if  $\theta_1+\theta_2+ \theta_3 > -n$, $\theta_1>-(n-1)$, and $\theta_3>-1$. In other words, $u_{\theta_1, \theta_2, \theta_3} \in L^1_{\rm{loc}}(\rn)$ if and only if $(\theta_1, \theta_2, \theta_3) \in \mathcal{A}\cup\mathcal{B}.$
\end{lem}

\begin{proof}
Similarly to the argument used in Lemma~\ref{lemma:L1:w}, we fix $R>0$ and split the integral of $u_{\theta_1, \theta_2, \theta_3}$ over the ball $B(0,R)$ in two cases: when $|x'|\le|x_n|$, and hence  $|x|\approx~|x_n|$ and  $u_{\theta_1, \theta_2, \theta_3}(x)\approx |x'|^{\theta_1}|x_n|^{\theta_2+\theta_3}$, or $|x_n|\le|x'|$, getting now that $|x|\approx~|x'|$ and  $u_{\theta_1, \theta_2, \theta_3}(x)\approx |x'|^{\theta_1+\theta_2}|x_n|^{\theta_3}$. Thus, we need to characterize the exponents $\theta_1, \theta_2, \theta_3$ for which these two terms are finite:
$$
I=\int_{\{x\in\mathbb R^n:|x|\le R,\ |x'|\le|x_n|\}}|x'|^{\theta_1}|x_n|^{\theta_2+\theta_3}\,dx\text{ and } II=\int_{\{x\in\mathbb R^n:|x|\le R,\ |x_n|\le|x'|\}}|x'|^{\theta_1+\theta_2}|x_n|^{\theta_3}\,dx.
$$
Now
\begin{align*}
I\approx\int_0^R\bigg(\int_{\{|x'|<t\}}|x'|^{\theta_1}\,dx'\bigg)t^{\theta_2+\theta_3}\,dt\approx\int_0^R\bigg(\int_0^t\rho^{\theta_1+n-2}\,d\rho\bigg)t^{\theta_2+\theta_3}\,dt,
\end{align*}
and the inner integral is finite if and only if $\theta_1>-(n-1)$ and we get
\begin{equation}\label{firstloc}
I\approx\int_0^Rt^{\theta_1+ n-1+\theta_2+\theta_3}\,dt<\infty\iff \theta_1>-(n-1)\text{ and }  \theta_1+\theta_2+\theta_3>-n.
\end{equation}
Similarly,
\begin{align*}
II\approx\int_{\{|x'|<R\}}\bigg(\int_0^{|x'|}t^{\theta_3}\,dt\bigg)|x'|^{\theta_1+\theta_2}\,dx'
\end{align*}
and the inner integral is finite if and only if $\theta_3>-1$ and we get
$$
II\approx\int_{\{|x'|<R\}}|x'|^{\theta_1+\theta_2+\theta_3+1}\,dx'<\infty\iff \theta_3>-(n-1)\text{ and }  \theta_1+\theta_2+\theta_3>-n,
$$
which, together with \eqref{firstloc}, proves de result.
\end{proof}

The next result comes from \cite{MZarXiv1} and identifies the triples of exponents $(\theta_{1},\theta_{2},\theta_{3})$ so that $u_{\theta_1, \theta_2, \theta_3} \in A_p(\rn)$. Namely, 

\begin{thm}[Theorem 3.2 from \cite{MZarXiv1}]\label{thm:Ainf:indices} Fix $1 < p < \infty$. Then, $u_{\theta_1, \theta_2, \theta_3} \in A_p(\rn)$ if and only if $(\theta_{1},\theta_{2},\theta_{3})\in (\mathcal{A}\cup\mathcal{B})\cap(\mathcal{C}_{p}\cup\mathcal{D}_{p})$.
\end{thm}

\section{A unifying approach towards weighted inequalities}\label{sec:approach}

Our strategy for the proofs of Theorems \ref{thm:t/s} and \ref{thm:main} will be based on Theorems \ref{coro:2.10:DD23} and \ref{thm:2.14:DD23} below, which provide an abstract framework for weighted interpolation inequalities. In particular, Theorem \ref{coro:2.10:DD23} represents an improvement on \cite[Corollary 2.10]{DD23}, whereas Theorem \ref{thm:2.14:DD23} is precisely Theorem 2.14 from \cite{DD23} whose statement has been included for the reader's convenience. The aforementioned improvement in Theorem \ref{coro:2.10:DD23} relies on the following lemma. 

\begin{lem}\label{lemma:GammaRS} Fix $S \subset \sn$ with $\mathcal{H}^{n-1}(S) > 0$,  $R> 0$, and define 
$$
\Gamma_R^S:=\{x \in \rn: |x| \geq R \text{ and } x/|x| \in S\}.
$$ 
Then, $w(\Gamma_R^S)  = \infty$ for every $w\in A_\infty(\rn)$.
\end{lem}

\begin{proof} It is well known that given $u \in A_\infty(\rn)$ there exist constants $0 < C_1 \leq C_2 $ and $0 < \theta_2 \leq \theta_1$ (depending only on $[u]_{A_\infty(\rn)}$ and $n$) such that 
\begin{equation}\label{u:vs:L}
C_1 \left(\frac{|E|}{|B|} \right)^{\theta_1} \leq  \frac{u(E)}{u(B)} \leq C_2 \left(\frac{|E|}{|B|} \right)^{\theta_2} 
\end{equation}
for every ball $B$ and every measurable subset $E \subset B$ (see for instance the proof of Lemma~9.2.1 as well as Theorem 9.3.3(d) from \cite{GrafakosBook}). Fix $\alpha > -n$ and set $w_\alpha(x):=|x|^\alpha$ so that $w_\alpha \in A_\infty(\rn)$ (see \cite[Example 9.1.7]{GrafakosBook}). The second inequality in \eqref{u:vs:L} applied to $w_\alpha$ gives
\begin{equation}\label{walpha:L}
 \frac{w_\alpha(E)}{w_\alpha(B)} \leq C_2 \left(\frac{|E|}{|B|} \right)^{\theta_2}
\end{equation}
for some $C_2, \theta_2$ depending only on $\alpha$ and $n$. Now, given $w \in A_\infty$, the first inequality in \eqref{u:vs:L} applied to $w$ gives
\begin{equation}\label{w:L}
C_1 \left(\frac{|E|}{|B|} \right)^{\theta_1} \leq  \frac{w(E)}{w(B)}
\end{equation}
where $C_1, \theta_1$ depend only on $[w]_{A_\infty(\rn)}$ and $n$. By combining \eqref{walpha:L} and \eqref{w:L} we obtain constants $C > 0$ and $\theta> 0$ such that
\begin{equation}\label{w:walpha}
\frac{w(E)}{w(B)} \geq C \left(\frac{w_\alpha(E)}{w_\alpha(B)}\right)^\theta
\end{equation}
for every ball $B$ and every measurable subset $E \subset B$. On the other hand, the second inequality in \eqref{u:vs:L} applied to $w$ and to concentric balls $B(x, r) \subset B(x, s)$ for $x \in \rn$ and $0 < r < s$ gives
\begin{equation}\label{RD:w}
\frac{w(B(x, s))}{w(B(x, r))} \geq c_1 \left(\frac{s}{r} \right)^{\eps_1}
\end{equation}  
for $c_1, \eps_1 > 0$ depending only on $[w]_{A_\infty(\rn)}$ and $n$. Now, given $S \subset \sn$ with $\mathcal{H}^{n-1}(S) > 0$, $R> 0$, and $R_1 > R$ set 
$$
E_{R, R_1} := \Gamma_R^S \cap B(0, R_1),
$$
so that \eqref{w:walpha} applied to $E_{R, R_1}$ and $B(0, R_1)$ implies
\begin{equation}\label{w:E:R:R1}
\frac{w(E_{R, R_1})}{w(B(0, R_1))} \geq C \left(\frac{w_\alpha(E_{R, R_1})}{w_\alpha(B(0, R_1))}\right)^\theta.
\end{equation} 
By changing to polar coordinates we can write
\begin{align*}
w_\alpha(E_{R, R_1}) = \int_{E_{R, R_1}} |x|^\alpha \dx = \mathcal{H}^{n-1}(S) \int_R^{R_1} \rho^{\alpha + n - 1} \, d\rho = \frac{\mathcal{H}^{n-1}(S)}{\alpha + n}(R_1^{\alpha + n} - R^{\alpha + n})
\end{align*}
and 
\begin{align*}
w_\alpha(B(0, R_1)) = \int_{B(0, R_1)} |x|^\alpha \dx = \mathcal{H}^{n-1}(\sn) \int_0^{R_1} \rho^{\alpha + n - 1} \, d\rho = \frac{\mathcal{H}^{n-1}(\sn)}{\alpha + n} R_1^{\alpha + n}.
\end{align*}
Thus, 
\begin{equation}\label{walpha:ERR1}
\left(\frac{w_\alpha(E_{R, R_1})}{w_\alpha(B(0, R_1))}\right)^\theta = \left(\frac{\mathcal{H}^{n-1}(S)}{\mathcal{H}^{n-1}(\sn)}\right)^\theta \left(1 - \left( \frac{R}{R_1}\right)^{\alpha +n} \right)^\theta.
\end{equation}
Finally, from the inequality \eqref{RD:w} used with $x=0$ and $0 < R < R_1$ we obtain 
$$
\frac{w(B(0, R_1))}{w(B(0, R))} \geq c_1 \left(\frac{R_1}{R} \right)^{\eps_1},
$$
which combined with \eqref{w:E:R:R1} and \eqref{walpha:ERR1} yields
\begin{align*}
w(E_{R, R_1}) & \geq C w(B(0, R_1)) \left(\frac{w_\alpha(E_{R, R_1})}{w_\alpha(B(0, R_1))}\right)^\theta\\
& \geq c_1 C \left(\frac{R_1}{R} \right)^{\eps_1} w(B(0,R)) \left(\frac{w_\alpha(E_{R, R_1})}{w_\alpha(B(0, R_1))}\right)^\theta \\
&= c_1 C  \left(\frac{R_1}{R} \right)^{\eps_1} w(B(0,R)) \left(\frac{\mathcal{H}^{n-1}(S)}{\mathcal{H}^{n-1}(\sn)}\right)^\theta \left(1 - \left( \frac{R}{R_1}\right)^{\alpha +n} \right)^\theta,
\end{align*}
and $w(\Gamma_R^S)  = \infty$ follows by letting $R_1 \to \infty$.\end{proof}

\begin{thm}[Case $\theta =1$, weighted Sobolev embedding]\label{coro:2.10:DD23} Let $1 < p \leq q < \infty$, $t <s < t +n$. Suppose that $(v,w) \in A_{p,q}^{s-t}$, and that $w, v^{-p'/p} \in \Ainf$. Then,
$$
\norm{D^t f}{L^q(w)} \lesssim \norm{D^s f}{L^p(v)}, \quad \forall f \in D^{>\tau} \mathcal{S}(\rn),
$$
with $\tau:= \max\{-n, -t\}$ and 
$
D^{> \tau}:=\{f \in L^\infty: f = D^bg\text{ for some } g \in \mathcal{S}(\rn) \text{ and } b > \tau\}.
$
\end{thm}

\begin{proof} The conclusion of Theorem \ref{coro:2.10:DD23} has been proved in \cite[Corollary 2.10]{DD23} under the additional assumption that $w(\Gamma_R^S) = \infty$ for every $S \subset \sn$ with $\mathcal{H}^{n-1}(S) > 0$ and  $R> 0$. Now, in view of Lemma \ref{lemma:GammaRS}, such assumption is automatically fulfilled by the fact that $w \in \Ainf$. \end{proof}

\begin{thm}[Theorem 2.14 from \cite{DD23}, case $\theta=t/s$]\label{thm:2.14:DD23} Let $1 < p, q, r < \infty$ and $0 < t < s$ be related by
\begin{equation}\label{rel:index:thmDD23}
\frac{1}{r}= \frac{t}{sp} +\left(1 -\frac{t}{s}\right)\frac{1}{q}.
\end{equation}
Let $\Sigma:= \{z \in \com: 0 < \mathcal{R}(z) < 1\}$ and suppose there is a family of weights $\{w_z\}_{z \in \overline{\Sigma}}$ satisfying the following conditions:
\begin{enumerate}[(i)]
\item\label{cond1:2.14} there is $h \in L^1_{\rm{loc}}(\rn)$ such that $|w_z| \leq h(x)$ for every $x \in \rn$ and $z \in \overline{\Sigma}$,
\item\label{cond2:2.14} for a.e. $x \in \rn$ the function $z \mapsto w_z(x)$ is continuous in $\overline{\Sigma}$ and analytic in $\Sigma$,
\item\label{cond3:2.14} and there are weights $w_0, w_1$ such that $|w_{i\tau}(x)| \leq w_0(x)$ and $|w_{1+i\tau}(x)| \leq w_1(x)$ for every $x \in \rn$ and $\tau \in \re$.
\end{enumerate}
Moreover, assume that there are weights $u, v, w$ such that  $w=w_{t/s}$, $(u^p, w_1^p) \in A_{p,p}$, $(v^q, w_0^q) \in A_{q,q}$, and $u^{-p'}, w_1^{p}, v^{-q'}, w_0^q \in \Ainf$. Then,
$$
\norm{w D^t f}{L^r} \lesssim \norm{u D^s f}{L^p}^{t/s} \norm{vf}{L^q}^{1-t/s}, \quad \forall f \in \mathcal{S}(\rn). 
$$
\end{thm}

\section{Proof of Theorem \ref{thm:t/s}}\label{sec:proof:thm:t/s}

Fix $1 < p, q, r < \infty$ and $0 < t < s$ related by \eqref{rel:index:thmDD23}. For $z \in \com$ with $0~\leq~\mathcal{R}(z)~\leq~1$ define
\begin{equation*}
w_z(x) := |x'|^{z \alpha_1 + (1-z) \alpha_0} |x|^{z \beta_1 + (1-z) \beta_0}|x_n|^{z \gamma_1 + (1-z) \gamma_0}
\end{equation*}
so that
\begin{equation*}
|w_z(x)| \leq |x'|^{\mathcal{R}(z) \alpha_1 + (1-\mathcal{R}(z)) \alpha_0} |x|^{\mathcal{R}(z) \beta_1 + (1-\mathcal{R}(z)) \beta_0} |x_n|^{\mathcal{R}(z) \gamma_1 + (1-\mathcal{R}(z)) \gamma_0}.
\end{equation*}
In particular, if $z= i\tau$ with $\tau \in \re$, then
\begin{equation*}
|w_{i \tau}(x)| \leq w_0(x) =  |x'|^{\alpha_0} |x|^{\beta_0} |x_n|^{\gamma_0}
\end{equation*}
and if $z=1+ i \tau$ with $\tau \in \re$
\begin{equation*}
|w_{1+i\tau}(x)|\leq w_1(x)= |x'|^{\alpha_1} |x|^{\beta_1}|x_n|^{\gamma_1}.
\end{equation*}
Let us check that the hypotheses from Theorem \ref{thm:2.14:DD23} hold true with the choice $u:=~w_1$ and $v:=w_0$. From the definition of the classes $A_{p,q}$ in \eqref{def:Apq},  notice that $(u^p, w_1^p) = (w_1^p, w_1^p)  \in A_{p,p}$ means $w_1^p \in A_p(\rn)$, which by \eqref{w:Ap:Ainf} is equivalent to $w_1^p, w_1^{-p'} \in \Ainf$. Now, since
$
w_1^p(x):= |x'|^{\alpha_1p} |x|^{\beta_1p}|x_n|^{\gamma_1p},
$
by virtue of Theorem \ref{thm:Ainf:indices} $w_1^p \in A_p(\rn)$ amounts to $p (\alpha_1, \beta_1, \gamma_1) \in (\mathcal{A}\cup\mathcal{B})\cap(\mathcal{C}_{p}\cup\mathcal{D}_{p})$, which is hypothesis \eqref{p:1:p}. Similarly, the hypothesis \eqref{q:0:q}, that is,  $q (\alpha_0, \beta_0, \gamma_0) \in (\mathcal{A}\cup\mathcal{B})\cap(\mathcal{C}_{q}\cup\mathcal{D}_{q})$ guarantees $w_0^q \in A_q(\rn)$ and therefore $(v^q, w_0^q)=(w_0^q, w_0^q) \in A_{q,q}$ as well as $w_0^{-q'}, w_0^q \in \Ainf$.  Finally, let us find $h \in L^1_{\rm{loc}}(\rn)$ such that $|w_z| \leq h$ for every $z \in \overline{\Sigma}$. By Young's inequality, 
\begin{align*}
 |x'|^{\mathcal{R}(z) \alpha_1 + (1-\mathcal{R}(z)) \alpha_0} & \leq |x'|^{\alpha_1} + |x'|^{\alpha_0} =:h_\alpha(x),\\
  |x|^{\mathcal{R}(z) \beta_1 + (1-\mathcal{R}(z)) \beta_0} & \leq |x|^{\beta_1} + |x|^{\beta_0}=:h_\beta(x),\\
 |x_n|^{\mathcal{R}(z) \gamma_1 + (1-\mathcal{R}(z)) \gamma_0} & \leq |x_n|^{\gamma_1} + |x_n|^{\gamma_0}=:h_\gamma(x). 
 \end{align*}
Thus, $|w_z| \leq h_\alpha h_\beta h_\gamma =:h$ with $h \in L^1_{\rm{loc}}(\rn)$ due to Lemma \ref{lemma:L1loc:indices} and the hypothesis \eqref{aj:bk:gl:L1loc}. All the hypotheses of Theorem \ref{thm:2.14:DD23} are then met and we obtain
\begin{equation*}
\norm{w_{t/s} D^t f}{L^r} \lesssim \norm{w_1 D^s f}{L^p}^{t/s} \norm{w_0 f}{L^q}^{1-t/s}, \quad \forall f \in \mathcal{S}(\rn), 
\end{equation*}
which is precisely \eqref{GN:theta:t/s}. \qed

\section{Proof of Theorem \ref{thm:theta=1}}\label{secc:theta:1} 

We will use Theorems \ref{thm:SW:Ainf:p:q:alpha0=0}, \ref{thm:SW:Ainf:p:q}, and \ref{coro:2.10:DD23} with $w :=u_{\theta_1, \theta_2, \theta_3}$ and $v:= u_{\sigma_1, \sigma_2, \sigma_3}$.  Indeed, by Theorem \ref{thm:Ainf:indices}, the hypothesis \eqref{sigma:theta:Ainf} means that $u_{\theta_1, \theta_2, \theta_3}, u_{\sigma_1, \sigma_2, \sigma_3}^{-p'/p} \in \Ainf$, that is, $w, v^{-p'/p} \in \Ainf$. Next, we check that $(v,w) \in A_{p,q}^{\alpha}$ where $\alpha := s-t \in (0,n)$. Indeed, if $\Theta_{p,q}(s-t) = 0$, then \eqref{theta:j:q:sigma:j:p} yields $w^{1/q} = v^{1/p}$ which together with Theorem \ref{thm:SW:Ainf:p:q:alpha0=0} (with $C_0 =1$) implies $[(v, w)]_{A_{p,q}^{s-t}} \leq  C_1$, where $C_1 > 0$ depends only on $n$, $p$, $q$, $s$, $t$, $\theta_1, \theta_2,$ $\theta_3, \sigma_1, \sigma_2, \sigma_3$. If $\Theta_{p,q}(s-t) > 0$, let us see that \eqref{sum:sigmas:thetas:alpha} yields $w^{1/q}/v^{1/p} \in L^{n/\Theta_{p,q}(s-t), \infty}(\rn)$ and then \eqref{wv:pq:thm:SW} will follow from Theorem \ref{thm:SW:Ainf:p:q}. Notice that $w^{1/q}/v^{1/p} \in L^{n/\Theta_{p,q}(s-t), \infty}(\rn)$ is equivalent to $\left(w^{1/q}/v^{1/p}\right)^{n/\Theta_{p,q}(s-t)} \in L^{1, \infty}(\rn)$ and by Lemma \ref{lemma:L1:w} we have that 
$$
\left(w^{1/q}/v^{1/p}\right)^{n/\Theta_{p,q}(s-t)} (x) = (|x'|^{\theta_1/q - \sigma_1/p} |x|^{\theta_2/q - \sigma_2/p} |x_n|^{\theta_3/q - \sigma_3/p})^{n/\Theta_{p,q}(s-t)} \in L^{1, \infty}(\rn),
$$
if and only if
\begin{align*}
\frac{n}{\Theta_{p,q}(s-t)}\left(\frac{\theta_1+ \theta_2 + \theta_3}{q}- \frac{\sigma_1+ \sigma_2 + \sigma_3}{p} \right) = -n,\\
 \frac{n}{\Theta_{p,q}(s-t)} \left(\frac{\theta_1}{q }- \frac{\sigma_1}{p}\right) > -(n-1),\\
 \frac{n}{\Theta_{p,q}(s-t)} \left(\frac{\theta_3}{q }- \frac{\sigma_3}{p}\right) > -1,
\end{align*}
which are precisely the conditions \eqref{sum:sigmas:thetas:alpha}, \eqref{sum:sigmas:thetas:alpha:2}, and \eqref{sum:sigmas:thetas:alpha:3}.  \qed

\section{Proof of Theorem \ref{thm:main}}\label{sec:inter:theta} 

The proof is a combination of Theorems \ref{thm:t/s} and \ref{thm:theta=1}.  By Theorem \ref{thm:t/s} applied  to $1 < p, q, a <\infty$ and $0 < \theta s < s$, so that the hypothesis \eqref{rel:apqtheta} becomes \eqref{rel:pqrst} (think ``$r=s$'' and ``$t=\theta s$'' in the notation from Theorem~\ref{thm:t/s}) and since $\alpha_0, \alpha_1, \beta_0,  \beta_1, \gamma_0, \gamma_1 \in \re$ satisfy \eqref{p:1:p}, \eqref{q:0:q}, and  \eqref{aj:bk:gl:L1loc} we get
\begin{equation}\label{GN:theta:t/s:2}
\norm{u_{\alpha_\theta, \beta_\theta, \gamma_\theta} D^{\theta s} f}{L^a} \lesssim \norm{u_{\alpha_1, \beta_1, \gamma_1} D^s f}{L^p}^{\theta} \norm{u_{\alpha_0, \beta_0, \gamma_0} f}{L^q}^{1-\theta}, \quad \forall f \in \mathcal{S}(\rn),
\end{equation}
where $(\alpha_\theta, \beta_\theta, \gamma_\theta) :=  (\alpha_1, \beta_1, \gamma_1) \theta +  (\alpha_0, \beta_0, \gamma_0)(1-\theta)$. 

Next, the idea is to apply Theorem \ref{thm:theta=1} with $1 < a \leq r < \infty$ and $t < \theta s < t +n$ (think ``$p=a$'', ``$q=r$'', and ``$s=\theta s$'' in the notation from Theorem \ref{thm:theta=1}) to obtain the inequality
\begin{equation}\label{aux:2}
\norm{u_{\theta_1/r, \theta_2/r, \theta_3/r} D^t f}{L^r} \lesssim \norm{u_{\sigma_1/a, \sigma_2/a, \sigma_3/a} D^{\theta s} f}{L^a},
\end{equation}
(recall Remark \ref{rmk:weights:inside}) which then determines the $\sigma_j$'s as 
\begin{equation}\label{sigmas:thetas}
(\sigma_1, \sigma_2, \sigma_3):= a (\alpha_\theta, \beta_\theta, \gamma_\theta),
\end{equation}
so that $\norm{u_{\alpha_\theta, \beta_\theta, \gamma_\theta} D^{\theta s} f}{L^a} = \norm{u_{\sigma_1/a, \sigma_2/a, \sigma_3/a} D^{\theta s} f}{L^a}$ and \eqref{aux:2} can be linked with \eqref{GN:theta:t/s:2}.

If $\Theta_{a,r}(\theta s - t) = 0$, we use Theorem \ref{thm:theta=1}\eqref{Theta =0} since the hypothesis \eqref{cond:delta:gamma:1:2:B:1} gives
$$
-\frac{a'}{a}(\sigma_1, \sigma_2, \sigma_3) = - a' (\alpha_\theta, \beta_\theta, \gamma_\theta)  \in \mathcal{A}\cup\mathcal{B}
$$
 as well as
 $$
 (\theta_1, \theta_2, \theta_3) := \frac{r}{a} (\sigma_1, \sigma_2, \sigma_3)=r (\alpha_\theta, \beta_\theta, \gamma_\theta)  \in \mathcal{A}\cup\mathcal{B},
 $$
 and \eqref{aux:2} follows from Theorem \ref{thm:theta=1}\eqref{Theta =0}. Thus, by combining \eqref{GN:theta:t/s:2} and \eqref{aux:2}, the inequality  \eqref{GN:E:0} is proved since $(\alpha_\theta, \beta_\theta, \gamma_\theta) = \frac{1}{r}(\theta_1, \theta_2, \theta_3)$ from the above definition of  $(\theta_1, \theta_2, \theta_3)$.
 
If $\Theta_{a,r}(\theta s - t) > 0$, given $\alpha_\theta', \beta_\theta', \gamma_\theta'$ as in \eqref{sum:sigmas:thetas:alpha:a:r}, \eqref{sum:sigmas:thetas:alpha:a:r:2} and \eqref{sum:sigmas:thetas:alpha:a:r:3} now define
$$
 (\theta_1, \theta_2, \theta_3) := r (\alpha_\theta', \beta_\theta', \gamma_\theta')
$$
and use Theorem \ref{thm:theta=1}\eqref{Theta >0} with \eqref{sum:sigmas:thetas:alpha:a:r}, \eqref{sum:sigmas:thetas:alpha:a:r:2} and \eqref{sum:sigmas:thetas:alpha:a:r:3}   playing the role of \eqref{sum:sigmas:thetas:alpha}, \eqref{sum:sigmas:thetas:alpha:2}, and \eqref{sum:sigmas:thetas:alpha:3} to obtain \eqref{aux:2}, always with the $\sigma_j$'s as in \eqref{sigmas:thetas}. Hence, \eqref{GN:E:1} follows from  \eqref{GN:theta:t/s:2} and \eqref{aux:2}.  \qed

\end{document}